\documentclass[11pt]{amsart}

\usepackage{amsfonts,amssymb,amsmath}
\usepackage{graphicx,tikz}
\usepackage{caption}
\usepackage{natbib}

\usepackage[margin=1.25in]{geometry}

\def\N{\mathbb{N}}
\def\Z{\mathbb{Z}}

\newtheorem{theorem}{Theorem}[section]
\newtheorem{proposition}[theorem]{Proposition}

\newtheorem{corollary}[theorem]{Corollary}

\newtheorem{example}[theorem]{Example}
\newtheorem{definition}[theorem]{Definition}

\begin{document}

\title{Geometric Characterization of Data Sets with Unique Reduced Gr\"{o}bner Bases}

\author{Elena S. Dimitrova \and Qijun He \and Brandilyn Stigler \and Anyu Zhang}


\maketitle

\begin{abstract}

Model selection based on experimental data is an important challenge in biological data science. Particularly when collecting data is expensive or time consuming, as it
is often the case with clinical trial and biomolecular experiments, the problem of selecting information-rich data becomes crucial for creating relevant models. We identify geometric properties of input data that result in a unique algebraic model and we show that if the data form a staircase, or a so-called linear shift of a staircase, the ideal of the points has a unique reduced Gr\"obner basis and thus corresponds to a unique model. We use linear shifts to partition data into equivalence classes with the same basis. We demonstrate the utility of the results  by applying them to a Boolean model of the well-studied
\textit{lac} operon in \textit{E. coli}. 

\medskip


\noindent\emph{MSC}: 14 \and 92
\end{abstract}

\section{Introduction}

Developing predictive models from large-scale experimental data is an important problem in biological data science \citep{schatz}.  
In certain settings, such as when modeling biological switches \citep{dalchau}, it is advantageous or even necessary to view the input data as discrete \citep{dimitrova2010}.
While there are many classes of functions that can fit data from an underlying network, data over a finite field can be fit by polynomials~\citep{lidl1997finite}. A useful consequence is that polynomial models of such data can be written in terms of a monomial basis, where each choice of monomial basis provides a different prediction regarding network structure \citep{jarrah2007reverse}.  

More precisely, when the entries of a given data set are elements of a finite field, they can be viewed as an algebraic variety of a system of polynomials and the corresponding set of  polynomials which vanish on the points in the variety is a \textit{zero-dimensional ideal of points} \citep{cox}.  A popular computational tool for ideals of points is a \textit{Gr\"obner basis} (GB), which is a generalization of echelon forms of linear systems \citep{abbott2000computing, farr, dong}.  Studying GBs, in particular the number of different GBs for a given ideal of points, has impact on any application which depends on GB computation for model construction.

While algebraic-geometry literature is vast with regards to finding solutions of polynomial systems, to the best of our knowledge, the only result which connects directly geometric properties of points in a variety $V$ to the uniqueness of reduced GBs of an  ideal of points $I(V)$ for any monomial order is found in \citep{robbiano-unique-gb} (Theorem 4.8(b)): it implies that an ideal of points $I(V)$ and the ideal of the set complement of the points $I(V^C)$ have the same number of reduced GBs and thus for any monomial order, $I(V)$ has a unique reduced GB if and only if $I(V^C)$ has a unique reduced GB.  The uniqueness of reduced GBs has significance in the following two applications.

In \citep{laubenbacher2004computational} Gr\"obner bases were applied to the problem of model selection in systems biology.  They were introduced as a tool to select minimal models from a set of polynomial dynamical systems that fit discrete experimental data: for a given set of data points over a finite field, the ideal of points forms a coset representing the space of polynomial dynamical systems that fit the data and a minimal model is selected from the space by computing a reduced GB of the ideal and taking the normal forms of the model equations. While this provides an algorithmic solution to model selection, each choice of monomial order results in a different GB and thus, likely, in a different minimal polynomial dynamical system. To remedy this computational artifact,  \citet{dimitrova2014data} proposed a systematic way of adding new data points to an existing data set to ensure that the ideal of points has a unique reduced GB, yielding a unique minimal model. They also gave an algebraic characterization of the smallest set of points that need to be added. The method, however, involves solving a large system of polynomial equations even for small data sets. 

More recently, neural ideals were introduced in \citep{neuralring} as an algebraic object that can be used to better understand
the combinatorial structure of neural codes. A neural ideal has a special generating set, called its canonical
form, that encodes a minimal description of the so-called receptive field structure intrinsic to the neural code.
Also, for a given monomial order, a neural ideal is generated by its (reduced) Gr\"obner basis with respect to that monomial order. It was shown in \citep{gb} that for small dimensions, Gr\"obner basis computations are faster than canonical form ones and it is thus desirable to be able to identify neural ideals whose canonical forms are Gr\"obner bases. They proceeded to show that this is the case exactly when the neural ideal has a unique Gr\"obner basis. However, there is still no known condition on the neural codes themselves that guarantees that the corresponding neural ideal has a unique Gr\"obner basis. 

This paper addresses an important question in algebraic design of experiments: What properties should a discrete data set satisfy in order to uniquely identify an algebraic model?  
Since the possibility of having multiple models that correspond to a single data set arises from the fact that a polynomial ideal can have multiple reduced GBs, we approach this question by identifying properties of the input data that result in a unique reduced Gr\"obner basis for any monomial order. 

For $m$ data points and $n$ variables, an upper bound for the number of distinct reduced Gr\"obner bases for an ideal of points is $O(m^{2n(n-1)/(n+1)})$ \citep{babson2003hilbert}; however in the finite field setting, this upper bound will be much smaller.   From an algebraic geometry point of view, the problem is the following: find varieties $V$ in $\mathbb Z_p^n$ whose ideals of points $I(V)$ have unique reduced GBs for any monomial order.  Alternatively the problem  can be phrased as identifying varieties $V$ for which the universal Gr\"obner basis of $I(V)$ is itself a reduced Gr\"obner basis, or for which the Gr\"obner fan of $I(V)$ consists of a single cone.  

We define the concept of a \textit{linear shift}, originally introduced in \citep{he2016} for varieties and extended to rings in \citep{robbiano-unique-gb}. We show that linear shifts are equivalence relations on data sets in $\mathbb Z_p^n$ (where $\Z_p$ is the field of integers modulo~$p$) of fixed size.
We use linear shifts to partition data into equivalence classes with the same basis.  This impacts research in biological data science which uses discrete data, including data storage in the biomedical sciences \citep{adam} and model selection in functional genomics \citep{chamberlin}.

In Section \ref{sec-linear-shift}, we establish a geometric property of the points in $V\subseteq \Z_p^n$ which guarantees that the ideal of the points $I(V)$ has a unique reduced Gr\"obner basis, regardless of the monomial order: Corollary \ref{cor:shift} to the main result, Theorem~\ref{thm:staircase}. Finally we apply the main result to a Boolean model of the well-studied \textit{lac} operon in \textit{E. coli}.

\section{Background}
\label{sec:background}

Much of the notation and formalization in this section is due to \citep{cox}. 

Let $k$ be a field and $n\in \N$. Let $I\subseteq k[x_1,\ldots,x_n]$ be an ideal and $LT_\prec(I)$ be the leading term ideal of $I$ with respect to some monomial order $\prec$.  Recall that the quotient ring $k[x_1,\ldots,x_n]/I$ is isomorphic to $span\{x^\alpha | x^\alpha \not \in  LT_\prec(I)\}$ as a $k$-vector space.  In fact, for each choice of monomial order, $\{x^\alpha | x^\alpha \not \in LT_\prec(I)\}$ forms a basis for $k[x_1,\ldots,x_n]/I$.  Such monomials are called \textit{standard} with respect to $\prec$; we denote the set $\{x^\alpha | x^\alpha \not \in LT_\prec(I)\}$ as $SM_\prec(I)$. Note that standard monomials satisfy the following divisibility property: if $x^\alpha \in SM_\prec(I)$ and $x^\beta | x^\alpha$, then $x^\beta \in SM_\prec(I)$. 

In the current setting, as all ideals are zero-dimensional, the quotient ring $k[x_1,\ldots,x_n]/I$ is finite dimensional as a vector space.  Hence the set $SM_\prec(I)$ of standard monomials associated with $I$ is finite and can be represented as a \textit{staircase}, which is a set $\lambda \subseteq \N^n$ of nonnegative integer vectors such that if $v \in \lambda$ and $u\leq v$ coordinate-wise, then $u\in \lambda$. Such staircases can be visualized on an integer lattice where a monomial is depicted via its exponent vector: $(m,n)\leftrightarrow x^my^n$; see Figure \ref{figure:standard_staircases}.   

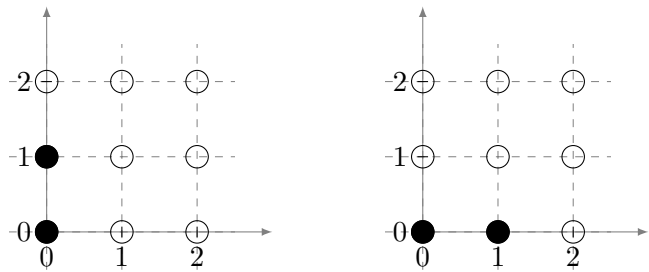
\begin{figure}[ht]
  \centering
  \begin{tikzpicture}
  \begin{scope}
    \coordinate (Origin)   at (0,0);
    \coordinate (XAxisMin) at (-0.5,0);
    \coordinate (XAxisMax) at (3,0);
    \coordinate (YAxisMin) at (0,-0.5);
    \coordinate (YAxisMax) at (0,3);
    \draw [thin, gray,-latex] (Origin) -- (XAxisMax);
    \draw [thin, gray,-latex] (Origin) -- (YAxisMax);

    \clip (-0.5,-0.5) rectangle (2.5cm,2.5cm); 
    \pgftransformcm{1}{0}{0}{1}{\pgfpoint{0cm}{0cm}}
    \coordinate (Bone) at (0,2);
    \coordinate (Btwo) at (2,-2);
    \draw[style=help lines,dashed] (-14,-14) grid[step=1cm] (14,14);
    \node[draw,circle,inner sep=3pt,fill] at (0,0) {};
      \node[draw,circle,inner sep=3pt,fill] at (0,1) {};
    \foreach \x in {0,1,...,7}{
      \foreach \y in {0,1,...,7}{
        \node[draw,circle,inner sep=3pt] at (\x,\y) {};
      }
    }
    \foreach \x in {0,1,2,3,4}
    \draw (\x cm,2pt) -- (\x cm,-2pt) node[anchor=north] {$\x$};
    \foreach \y in {0,1,2,3,4}
    \draw (2pt,\y cm) -- (-2pt,\y cm) node[anchor=east] {$\y$};
\end{scope}
    
     \begin{scope}[shift={(5,0)}]
    \coordinate (Origin)   at (0,0);
    \coordinate (XAxisMin) at (-0.5,0);
    \coordinate (XAxisMax) at (3,0);
    \coordinate (YAxisMin) at (0,-0.5);
    \coordinate (YAxisMax) at (0,3);
    \draw [thin, gray,-latex] (Origin) -- (XAxisMax);
    \draw [thin, gray,-latex] (Origin) -- (YAxisMax);

    \clip (-0.5,-0.5) rectangle (2.5cm,2.5cm); 
    \pgftransformcm{1}{0}{0}{1}{\pgfpoint{0cm}{0cm}}
    \coordinate (Bone) at (0,2);
    \coordinate (Btwo) at (2,-2);
    \draw[style=help lines,dashed] (-14,-14) grid[step=1cm] (14,14);
    \node[draw,circle,inner sep=3pt,fill] at (0,0) {};
     \node[draw,circle,inner sep=3pt,fill] at (1,0) {};
    \foreach \x in {0,1,...,7}{
      \foreach \y in {0,1,...,7}{
        \node[draw,circle,inner sep=3pt] at (\x,\y) {};
      }
    }
    \foreach \x in {0,1,2,3,4}
    \draw (\x cm,2pt) -- (\x cm,-2pt) node[anchor=north] {$\x$};
    \foreach \y in {0,1,2,3,4}
    \draw (2pt,\y cm) -- (-2pt,\y cm) node[anchor=east] {$\y$};
\end{scope}
  \end{tikzpicture}
  \caption{Staircases of the ideal $I(V)$, where $V=\left\{(2,0),(0,1)\right\}\subseteq \mathbb{Z}_3^2$. For a given Gr\"{o}bner
basis, each diagram shows the interface between the standard monomials, which are represented by black dots, and the leading terms, which are represented by white dots. Here, $SM_1(I(V))=\{1,y\}$ and $SM_2(I(V))=\{1,x\}$.}
  \label{figure:standard_staircases}
\end{figure}

Finally, recall that if $\prec$ is a monomial order and $I \subseteq k[x_1, \ldots, x_n]$ is nonzero, then a subset $G = \{g_1, \ldots, g_t\}$ is a \textit{Gr\"{o}bner basis} for $I$ with respect to $\prec$  if $\langle LT_\prec(g_1), \ldots , LT_\prec(g_t)\rangle = \langle LT_\prec(I)\rangle$. Furthermore, a Gr\"{o}bner basis is \textit{reduced} if the leading coefficient of each element of the basis is $1$ and no monomial in any element of the basis is in the ideal generated by the leading terms of the other elements of the basis. Given a fixed monomial order, an ideal $I$ has a unique reduced Gr\"{o}bner basis. 

The union of all reduced Gr\"obner bases is \textit{the universal Gr\"obner basis} of $I$, which is a Gr\"obner basis for every monomial order.  We note that the universal GB may not be reduced when $I$ has multiple reduced GBs associated to different monomial orders.

\begin{example}
\label{toy-ex}
Let $V=\{(0,0),(1,0),(2,1)\}\subseteq \mathbb Z_3$ and $I(V)\subseteq \mathbb Z_3[x,y]$ be the ideal of polynomials that vanish on the points in~$V$. Let~$\prec_1$ denote the monomial order with weight vector $(1,1)$ and~$\prec_2$ the order with weight vector $(1,3)$.
Then $I(V)$ has two distinct reduced Gr\"{o}bner bases, $GB_{\prec_1}(I(V))=\{y^2-y, xy+y, x^2-x+y\}$ and $GB_{\prec_2}(I(V)) =\{x^3-x, y+x^2-x\}$. So the universal Gr\"obner basis for $I(V)$ is $\mathcal G=\{y^2-y, xy+y, x^2-x+y\}\cup \{x^3-x,y+x^2-x\}$.  This set has two different leading term ideals: $LT_{\prec_1}(I(V))=\langle y^2, xy, x^2\rangle$ and  $LT_{\prec_2}(I(V))=\langle x^3-x,y+x^2-x\rangle$. The associated standard monomial bases are $SM_{\prec_1}(I(V)) =\{1, x, y\}$ and $SM_{\prec_2}(I(V)) = \{1, x, x^2\}$.
\end{example}

We note that every ideal of points over a finite field of characteristic~$p>0$ has  relations of the form~$x^p-x$ for every variable.  As such, the degrees of a variable in a polynomial are bounded above by~$p-1$, though polynomials may have high total degree due to many mixed terms.

\begin{example}
Consider a network of three nodes $x,y,z$ with states in $\mathbb Z_3$, where 0 stands for low, 1 for medium, 2 for high. Suppose that $z$ is regulated by~$x$ and~$y$ and it outputs are displayed in  the table below for the input data in Example~\ref{toy-ex}.
\begin{center}
\begin{tabular}{|l|| l l |c|}
\hline
Time & $x$ & $y$ & $z=f(x,y)$  \\ \hline
$t_1$ & 0 & 0 & 0 \\ \hline
$t_2$ & 1 & 0 & 0 \\ \hline
$t_3$ & 2 & 1 & 1 \\ \hline
\end{tabular}
\end{center}
As computed above, we know that $I(V)$ has two distinct reduced GBs with corresponding standard monomial bases.  Given each basis,~$f$ can be written as a linear combination $ax^{\alpha_1}+bx^{\alpha_2}+cx^{\alpha_3}$ of the basis elements.  To find the coefficients, one can set up evaluation matrices~$\mathbb X$ as follows:
\vspace{0.1in}

\begin{center}
\begin{tabular}{|l|l l l|}
\hline
 $(x,y)$ & $1$ & $x$ & $y$  \\ \hline
 (0,0) & 1 & 0 & 0  \\ \hline
 (1,0) & 1 & 1 & 0  \\ \hline
 (2,1) & 1 & 2 & 1  \\ \hline
\end{tabular}
\hspace{0.1in} and \hspace{0.1in}
\begin{tabular}{|l|l l l|}
\hline
 $(x,y)$ &$1$ & $x$ & $x^2$  \\ \hline
 (0,0) &1 & 0 & 0 \\ \hline
 (1,0) &1 & 1 & 1 \\ \hline
 (2,1) &1 & 2 & 1 \\ \hline
\end{tabular} 
\vspace{0.1in}
\end{center}
\noindent Solving each linear system $\mathbb X%
[\begin{matrix}
    a & b & c \\
\end{matrix}]^T
=[\begin{matrix}
    0 & 0 & 1 \\
\end{matrix}]^T
$
results in two candidate models for $f$ which are compatible with the given data: $f=y$ and $f=x + 2x^2$.

\end{example}

Note that the first function model is regulated only by gene $y$ but the second model is regulated only by gene $x$.  This presents a problem for experimentalists in that the two models suggest different regulatory relationships. We therefore ask the question: Which datasets generate a unique model?  That is, we are interested in  varieties for which the universal Gr\"obner basis is itself a reduced GB and has a single leading term ideal (and therefore single standard monomial basis) regardless of which monomial order is chosen.  

\section{Geometric properties of data sets whose ideals have unique reduced Gr\"{o}bner bases} 
\label{sec-linear-shift}

In this section we give a geometric characterization of $V\subseteq \Z_p^n$ such that $I(V)$ has a unique reduced Gr\"{o}bner basis, regardless of the monomial order. Unless otherwise stated, definitions and results in this section are found in~\citep{he2016}.

In Section \ref{sec:background}, we recalled that exponent vectors of standard monomials of an ideal $I$ form a staircase since any divisor of a standard monomial is again a standard monomial. Such a staircase is called \textit{initial} \citep{babson2003hilbert}; for example, the staircases in Figure \ref{figure:standard_staircases} are initial. Furthermore, a staircase $\lambda$ is \textit{basic}  for $I$ if the congruence classes modulo $I$ of the monomials $x^v$ with $v\in \lambda$ form a vector space basis for the quotient space $\Z_p [x_1,\ldots,x_n]/I$ \citep{babson2003hilbert}. 

If $\lambda$ is basic, then the class $[f] = f + I$ of any $f\in \Z_{p}[x_1,\ldots,x_n]$ can be uniquely represented as a linear combination of elements in $\{x^v\mid v\in \lambda\}$. For a given monomial order, any polynomial $f\in \Z_{p}[x_1,\ldots,x_n]$ has a unique normal form with respect to $I$. Hence an initial staircase of an ideal $I$ is basic. As it will be shown in Theorem~\ref{thm:basic}, we can determine whether a set $\lambda \in \Z_p^n$ is basic by checking the invertibility of the evaluation matrix defined next.

\begin{definition}
Let $\lambda=\{u^1,\ldots,u^r\}$ be an $r$-subset of $\Z_p^n$ and let $V=\{v^1,\ldots,v^s\}$ be an $s$-subset of $\Z_p^n$. The \textit{evaluation matrix} $\mathbb{X}(x^{\lambda},V)$ is the $s$-by-$r$ matrix whose element in position $(i,j)$ is $x^{u^j}(v^i)$, the evaluation of $x^{u^j}$ at $v^i$.
\end{definition}

\begin{example}
Let $\lambda_1=\{(0,0),(1,0)\}$, $\lambda_2=\{(0,0),(0,1)\}$, and $V=\{(2,0),(0,1)\}$ be subsets of $\Z_3^2$. Then $\mathbb{X}(x^{\lambda_1},V)=\left[\begin{array}{cc} 1&2\\ 1&0\end{array}\right]$ and  $\mathbb{X}(x^{\lambda_2},V)=\left[\begin{array}{cc} 1&0\\ 1&1\end{array}\right]$.
\end{example}

The following result from \citep{babson2003hilbert} illustrates a connection between basic sets and evaluation matrices, which will be used below.

\begin{theorem}[\citep{babson2003hilbert}]\label{thm:basic}
Let $\lambda$ and $V$ be subsets of $\Z_p^n$. Then $\lambda$ is basic for $I(V)$ if and only if $\mathbb{X}(x^{\lambda},V)$ is invertible. 
\end{theorem}

\begin{example}
Let $\lambda_1=\{(0,0),(1,0)\}$, $\lambda_2=\{(0,0),(0,1)\}$, and $V=\{(0,0),(1,0)\}$ be subsets of $\Z_3^2$. Then $\lambda_1$ is basic for $I(V)$ since $\mathbb{X}(x^{\lambda_1},V)=\left[\begin{array}{cc} 1&0\\ 1&1\end{array}\right]$ is invertible; however $\lambda_2$ is not basic for $I(V)$ since $\mathbb{X}(x^{\lambda_2},V)=\left[\begin{array}{cc} 1&0\\ 1&0\end{array}\right]$ is not invertible.
\end{example}

Notice that an initial staircase must be basic, while a basic staircase might not be initial; however, if $I(V)$ has a unique initial staircase (and thus a unique reduced Gr\"{o}bner basis), then $I(V)$ has a unique basic staircase.


\begin{proposition}\label{prop:basic}
An ideal $I(V)$ has a unique initial staircase if and only if $I(V)$ has a unique basic staircase.
\end{proposition}
\begin{proof}
We recall that the initial ideals of a polynomial ideal are in a one-to-one correspondence with its reduced Gr\"{o}bner bases. The result then follows directly from Proposition 2.2 in \citep{babson2003hilbert} and the fact that for every two monomials $x^{\alpha},x^{\beta}$ with $x^{\alpha} \nmid x^{\beta}$, there exists a weight vector~$\gamma$ and monomial order $\prec_{\gamma}$ such that $x^{\beta}~\prec_{\gamma} x^{\alpha}$.
\end{proof}
Based on Proposition \ref{prop:basic}, if we want to find out whether $I(V)$ has a unique reduced Gr\"{o}bner basis, we just need to check whether $I(V)$ has a unique basic staircase. In other words, we can check if there exist a unique staircase $\lambda \subseteq \Z_p^n$ such that $\mathbb{X}(x^{\lambda},V)$ is invertible.

Next we present a sufficient condition for $I(V)$ to have a unique reduced Gr\"{o}bner basis.  A construction that will aid in proving this condition is that of a \textit{layer} of a staircase. 

\begin{definition}
Given a staircase $\lambda=\{(u_1,\ldots,u_n):u_i\in \mathbb{N}\}$, the \textit{$i$th layer} of~$\lambda$ with respect to the $j$th coordinate is the subset $\{u\in \lambda:u_j=i\}\subseteq \lambda$. Let~$\ell$ be the largest integer such that $\{u\in \lambda:u_j=\ell\}\neq \emptyset$. The \textit{height} of $\lambda$ in the $jth$ coordinate is defined to be $\ell+1$, denoted as $h_j(\lambda)$.
\end{definition}

Note that we identify a layer of a staircase in $\N^n$ with a staircase in $\N^{n-1}$.

Proposition \ref{prop:basic} prompts the following natural question. If $\lambda\subseteq \Z_p^n$ is a staircase, which subsets $V\subseteq \Z_p^n$ have $\lambda$ as their \textit{unique} basic staircase? Let us consider the case when $V$ is itself a staircase. 

\begin{theorem}\label{thm:staircase}
Let $\lambda$ and $V$ be two staircases in $\Z_p^n$. Then $\lambda$ is basic for $I(V)$ if and only if $\lambda=V$.
\end{theorem}

\begin{proof}
We will prove this by induction on the number of variables.

\smallskip

Let $n=1$. 
To prove the necessary condition, suppose $\lambda\neq V$. Since there is only one variable, we must have $|\lambda|\neq |V|$. Therefore, $\mathbb{X}(x^{\lambda},V)$ is not invertible since it is not a square matrix. By Theorem \ref{thm:basic}, $\lambda$ is not basic.  
For the sufficient condition, assume $\lambda=V$. Then $\mathbb{X}(x^{\lambda},V)$ is a square Vandermonde matrix. Since $V$ is a set of distinct points, $\mathbb{X}(x^{\lambda},V)$ is invertible. By Theorem \ref{thm:basic}, $\lambda$ is basic.

Assume the inductive hypothesis holds for $n=k$ and consider $n=k+1$. We will prove the inductive step by induction on the height of the monomial staircase with respect to the first coordinate.  

Consider the base case when $h_1(\lambda)=1$; in other words, for all $u\in \lambda$, we have $u_1=0$. That is to say all monomials of $x^{\lambda}$ do not involve $x_1$. 

\smallskip
\begin{itemize}
\item[$(\Leftarrow)$] Suppose $\lambda=V$. Then the first coordinate of any point in $V$ is $0$. Therefore, $\lambda$ and $V$ are staircases with one fewer variable, $x_1$. Based on our inductive hypothesis for $n=k$, we have $x^{\lambda}$ are basic monomials of $I(V)$.

\item[$(\Rightarrow)$] Suppose $\lambda\neq V$. Then we consider three cases. 
\begin{itemize}
	\item[C1:] $|\lambda|\neq |V|$. In this case, $\mathbb{X} (x^{\lambda},V)$ is not invertible as it is not square.

	\item[C2:] $|\lambda|= |V|$ and $h_1(V)\geq 2$. In this case, $\mathbb{X} (x^{\lambda},V)$ is not invertible as at least two rows are the same.

	\item[C3:] $|\lambda|= |V|$ and $h_1(V)=h_1(\lambda)=1$, while $\lambda\neq V$. In this case, the first coordinate of any point in $\lambda$ and $V$ is $0$. In other words, $\lambda$ and $V$ are staircases with one fewer variable, $x_1$. Based our inductive hypothesis for $n=k$, we have that $\mathbb{X}(x^{\lambda},V)$ is not invertible.

\end{itemize}
In each case, the evaluation matrix is not invertible and so $\lambda$ is not basic, concluding the base case of $h_1(\lambda)=1$.
\end{itemize}

Assume the inductive hypothesis holds for all monomial staircases $\lambda$ with $1\leq h_1(\lambda) \leq d$. Let us consider a staircase $\lambda$ with $h_1(\lambda)=d+1$.

\begin{itemize}
\item[$(\Leftarrow)$] Suppose $\lambda=V$. Let $\lambda_0:=\{u\in \lambda :u_1=0\}$ denote the $0$th layer of $\lambda$ and $V_0:=\{v\in \lambda :v_1=0\}$ denote the $0$th layer of $V$ with respect to the first coordinate. Since $\lambda=V$, we have $\lambda_0=V_0$. By inductive hypothesis, the evaluation matrix $\mathbb{X}(x^{\lambda_0},V_0)$ is invertible. Now let us consider the evaluation matrix $\mathbb{X}(x^{\lambda},V)$. We can reorder rows and columns of $\mathbb{X}(x^{\lambda},V)$, so that $\mathbb{X}(x^{\lambda_0},V_0)$ appears as the upper left submatrix of $\mathbb{X}(x^{\lambda},V)$. Since the upper left submatrix of $\mathbb{X}(x^{\lambda},V)$ is invertible, after elementary row and column operations, we can transform $\mathbb{X}(x^{\lambda},V)$ into a block matrix of the form $$\left[\begin{array}{cc} I&0\\ 0&\mathbb{X}(x^{\lambda\backslash\lambda_0},V\backslash V_0)\end{array}\right].$$ 
Moreover, $x_1$ divides all monomials in $x^{\lambda\backslash\lambda_0}$ and for any point $v\in V\backslash V_0$, we have $v_1\neq 0$. Therefore, $\mathbb{X}(x^{\lambda\backslash\lambda_0},V\backslash V_0)$ is invertible if and only if $\mathbb{X}(\frac{x^{\lambda\backslash \lambda_0}}{x_1},V\backslash V_0)$ is invertible. Note that $\frac{x^{\lambda\backslash \lambda_0}}{x_1}$ corresponds to a staircase with height at most $d$ and $V\backslash V_0$ is a linear shift of the same staircase, so~$X^*$ is invertible by the inductive hypothesis and Theorem \ref{prop:shift}. Hence the original evaluation matrix $\mathbb{X}(x^{\lambda},V)$ is also invertible; thus $\lambda$ is basic.

\item[$(\Rightarrow)$] Suppose $\lambda\neq V$. If $|\lambda|\neq |V|$, then $\mathbb{X}(x^{\lambda},V)$ is not invertible since it is not a square matrix.

Assume $|\lambda|= |V|$. We can reorder rows and columns so that $\mathbb{X}(x^{\lambda},V)$ appears in the form
\[
\left[\begin{array}{cc} \mathbb{X}(x^{\lambda_0},V_0) &0\\ A&\mathbb{X}(x^{\lambda\backslash\lambda_0},V\backslash V_0)\end{array}\right].
\]
Note that the row space of $A$ is a subspace of the row space of $\mathbb{X}(x^{\lambda_0},V_0)$. 

Suppose $\mathbb{X}(x^{\lambda_0},V_0)$ is not a square matrix. We then consider two cases.
\begin{itemize}
	\item[C1:] If $\mathbb{X}(x^{\lambda_0},V_0)$ has more rows than columns, then the rows of $[\mathbb{X}(x^{\lambda_0},V_0),0]$ are linearly dependent.

	\item[C2:] If $\mathbb{X}(x^{\lambda_0},V_0)$ has more columns than rows, then the columns of $[\mathbb{X}(x^{\lambda_0},V_0),A]^T$ are linearly dependent.
\end{itemize}

In either case, $\mathbb{X}(x^{\lambda_0},V_0)$ is not invertible.

If $\mathbb{X}(x^{\lambda_0},V_0)$ is a square matrix but $\lambda_0\neq V_0$, then $\mathbb{X}(x^{\lambda_0},V_0)$ is not invertible by the inductive hypothesis. So $\mathbb{X}(x^{\lambda},V)$ is not invertible.

If $\mathbb{X}(x^{\lambda_0},V_0)$ is a square matrix and $\lambda_0= V_0$, then we must have $\lambda\backslash\lambda_0 \neq V\backslash V_0$. Note that $\mathbb{X}(x^{\lambda\backslash\lambda_0},V\backslash V_0)$ is invertible if and only if $\mathbb{X}(\frac{x^{\lambda\backslash \lambda_0}}{x_1},V\backslash V_0)$ is invertible. Since $\lambda\backslash\lambda_0 \neq V\backslash V_0$, $\frac{x^{\lambda\backslash \lambda_0}}{x_1}$ corresponds to a staircase with height at most $d$, and $V\backslash V_0$ is a linear shift of some other staircase, $\mathbb{X}(\frac{x^{\lambda\backslash \lambda_0}}{x_1},V\backslash V_0)$ is not invertible by the inductive hypothesis and Theorem~\ref{prop:shift}. Therefore, $\mathbb{X}(x^{\lambda},V)$ is also not invertible.

Since the evaluation matrix is not invertible, $\lambda$ is basic. 
\end{itemize}

Hence the inductive step holds for $h_1(\lambda)=d+1$, thereby concluding the inductive proof on the height of $\lambda$.  Completing the nested inductive proof completes the inductive step for $n=k+1$. Therefore, the original statement holds for all $n\in \N$.  
\end{proof}

\begin{corollary}\label{cor:UGB}
If $V\subseteq \Z_p^n$ is a staircase, then $I(V)$ has a unique reduced Gr\"{o}bner basis for any monomial order.
\end{corollary}

\begin{proof}
It follows immediately from Theorem \ref{thm:staircase}.
\end{proof}

\begin{example}
The set $V=\{(0, 0), (0, 1), (1, 0)\}$ is a staircase in $\Z_3^2$. Its corresponding ideal $I(V)$ has a unique reduced Gr\"{o}bner basis $\mathcal{G}=\{y^2-y, xy, x^2-x\}$ with respect to any monomial order.
\end{example}

Before we state the main theorem of this section, we introduce an equivalence relation on the subsets of $\Z_p^n$ that is useful in the context of staircases and Gr\"{o}bner bases.

\begin{definition}\label{linear-shift}
For $V_1,V_2\subseteq \Z_p^n$ with $|V_1|=|V_2|$, we say that $V_1$ is a \textit{linear shift} of $V_2$, denoted $V_1\overset{L}{\sim} V_2$, if there exists  $\phi=(\phi_1,\ldots,\phi_n):\Z_p^n \rightarrow \Z_p^n$ such that $V_1=\phi(V_2)$ and $\phi_i:\mathbb Z_p\rightarrow \mathbb Z_p$ is defined coordinate-wise as $\phi_i(x_i)=a_ix_i+b_i$ for some $a_i\in \left(\Z_p\right)^\times$ and $b_i\in \Z_p$ for $i=1,\ldots,n$.
\end{definition}

It is straightforward to see that linear shift is a bijection between two point sets and defines an equivalence relation on subsets of $\mathbb Z_p^n$ of the same size.

\begin{example}
Consider $V_1,V_2,V_3\subseteq \Z_3^2$, where $V_1=\{(0,0),(0,1)\}$, $V_2=\{(1,1),(1,2)\}$, and $V_3=\{(1,1),(2,2)\}$. Then $V_1\overset{L}{\sim} V_2$ since $V_1=\phi(V_2)$, where $\phi=(x+1,x+1)$. On the other hand, $V_1\not\overset{L}{\sim} V_3$ since the first coordinates of the points in $V_1$ are the same while in $V_3$ they are different.
\end{example}

We will see next that linear shifts preserve standard monomial bases.

\begin{theorem}\label{prop:shift}
If $V_1,V_2\subseteq \Z_p^n$ and $V_1\overset{L}{\sim} V_2$, then $I(V_1)$ and $I(V_2)$ have the same leading term ideals and thus the same standard monomial bases.
\end{theorem}

\begin{proof}
Let $V_1\overset{L}{\sim} V_2$.  Then there is a permutation $\phi(x_1,\ldots,x_n)=(a_1x_1+b_1,\ldots,a_nx_n+b_n)$ for some $a_i, b_i\in \mathbb Z_p$ with $a_i\neq 0$ for all $i$, such that $\phi(V_1)=V_2$. Observe that
\[
\begin{split}
I(V_2) &= \{f:f(v)=0, \forall v\in V_2\}\\
&= \{f:f(\phi(u))=(f\circ \phi )(u)=0,\forall u\in V_1\}.
\end{split}
\]
Thus, for any $f\in I(V_2)$, we have $f\circ \phi \in I(V_1)$. Since $V_2$ is a linear shift of $V_1$ via~$\phi$, $(f\circ \phi)(x_1,\ldots,x_n)=f(\phi(x_1,\ldots,x_n)=f(a_1x_1+b_1,\ldots,a_nx_n+b_n)$.  So $f$ and $f\circ \phi$ have the same leading monomial with respect to any monomial order~$\prec$. Therefore, $LT_\prec(I(V_2))\subseteq LT_\prec(I(V_1))$. We can show $LT_\prec(I(V_1))\subseteq LT_\prec(I(V_2))$ by replacing $\phi$ with $\phi^{-1}$. Hence $LT_\prec(I(V_1))= LT_\prec(I(V_2))$ with respect to any monomial order. 
\end{proof}


\begin{corollary}\label{cor:same}
If $V_1,V_2\subseteq \Z_p^n$ and $V_1\overset{L}{\sim} V_2$, then $I(V_1)$ and $I(V_2)$ have the same number of reduced Gr\"{o}bner bases. In particular, for any monomial order, $I(V_1)$ has a unique reduced Gr\"{o}bner basis if and only if $I(V_2)$ has a unique reduced Gr\"{o}bner basis.
\end{corollary}

\begin{proof}
Due to the one-to-one correspondence between initial ideals and reduced Gr\"{o}bner bases, Theorem~\ref{prop:shift} implies that $I(V_1)$ and $I(V_2)$ have the same number of reduced Gr\"{o}bner bases. 
\end{proof}

\begin{corollary} \label{cor:shift}
If $V\subseteq \Z_p^n$ is a linear shift of a staircase, then $I(V)$ has a unique reduced Gr\"{o}bner basis.
\end{corollary}

\begin{proof}
This follows from Theorems \ref{thm:staircase} and \ref{prop:shift}.
\end{proof}

\begin{example}
Consider subsets $V_1=\{(0, 0), (0, 1), (1, 0)\}$ and $V_2=\{(0, 1), (0, 2), (2, 2)\}$ of $\Z_3^2$. Notice that $V_1$ is a staircase while $V_2$ is not. As $V_2=\phi(V_1)$, where $\phi=(2x,2x+2)$, $V_1\overset{L}{\sim}V_2$. We see that $I(V_2)$ has a unique reduced Gr\"{o}bner basis $\mathcal{G}=\{y^2-1, xy+x, x^2+x\}$.
\end{example}

A set of points being a linear shift of a staircase is not a necessary condition for the corresponding ideal of points to have a unique reduced Gr\"{o}bner basis, as is shown in the following example.

\begin{example}
Consider the set $V=\{(0, 0, 0), (1, 0, 0), (1, 1, 0), (1, 1, 1)\}$.  Since each function in the component-wise definition of a linear shift $\phi$ is of the form $\phi_i(x_i)=x_i$ or $x_i+1$ for $i=1,2,3$, there are only $2^3$ possible linear shifts and by inspection we see that when applied to $V$ none of them results in a staircase. Therefore, $V$ is not a linear shift of a staircase in $\Z_2^3$. However, using Theorem~3.3(b) in \citep{robbiano-unique-gb}, for example, we see  that $I(V)$ has this single reduced Gr\"{o}bner basis $\mathcal{G}=\{z^2+z, yz+z, y^2+y, xz+z, xy+y, x^2+x\}$ regardless of the choice of monomial order.		
\end{example}

We conclude this section by noting that $V\subseteq \Z_2^n$ is a simplicial complex if and only if it is a staircase. Thus Corollaries \ref{cor:UGB} and \ref{cor:shift} generalize Proposition~4.2 in \citep{gb} for an arbitrary finite field and linear shift.

\section{Applications}

In this section, we apply our main results to a small Boolean model of the well-studied \textit{lac} operon and highlight a possible improvement on the time to compute multiple Gr\"obner bases.

\subsection{Design of Experiments in the Lac Operon}
\label{sec:lacop}

The \textit{lac} operon is a system of genes which control the transport and metabolism of lactose in many bacteria including \textit{E. coli}. While there are numerous models for the \textit{lac} operon (see, for example, \citep{lac-santillian,lac-goodwin,lac-ozbudak,lac-wong}), we consider a Boolean model proposed in \citep{vcs}. There the authors reduced the system to a core subnetwork consisting of the following four variables: $M$ representing \textit{lac} mRNA, $L$ intercellular lactose, $L_e$ extracellular lactose, and $G_e$ extracellular glucose.  The Boolean model for this subnetwork is given by the following Boolean functions, where extraneous variables are introduced to capture intermediate values ($L_m, L_{em}$) of lactose inside and outside of the cell respectively: see Section 4.2.2 in \citep{vcs} for a full description of the model.

\begin{align*}
f_M&=\neg G_e\wedge(L \vee L_m)\\
f_L&=M\wedge L_e\wedge \neg G_e\\
f_{L_m}&=((L_{em}\wedge M)\vee L_e)\wedge \neg G_e\\
f_{L_e}&=L_e\\
f_{G_e}&=G_e\\
f_{L_{em}}&=L_{em}
\end{align*}

For the sake of illustrating the utility of the above results, we reduce this model to only include the four essential variables. To this end, we replace $L_{em}$ with $L_e$ and $L_m$ with $L$, and remove all instances of $L_{em}$ and $L_m$ via substitution.  Doing so produces 
$$f_{L_m}=((L_{e}\wedge M)\vee L_e)\wedge \neg G_e=L_e \wedge \neg G_e$$ 
which we substitute into the function $f_M$:
$$f_M=\neg G_e\wedge(L \vee (L_e \wedge \neg G_e))=\neg G_e\wedge(L \vee L_e).$$
This results in the following Boolean network on four variables, with wiring diagram depicted in Figure~\ref{lacop:wd}:
\begin{align*}
f_M&=\neg G_e\wedge(L \vee L_e)\\
f_L&=M\wedge L_e\wedge \neg G_e\\
f_{L_e}&=L_e\\
f_{G_e}&=G_e.\\
\end{align*}

\begin{figure}[h]
    \centering
\includegraphics[width=5cm]{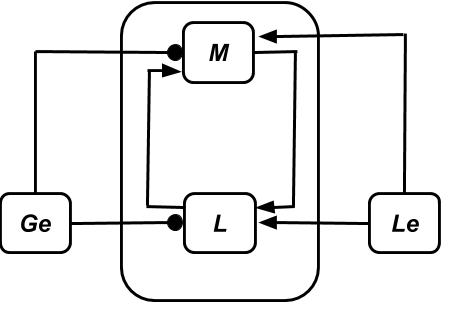}
    \caption{Wiring diagram for a simplified Boolean model of the \textit{lac} operon in \textit{E. coli}.  Directed edges with pointed ends indicate positive regulation, while directed edges with round ends indicate negative regulation.  The variables $G_e$ and $L_e$ regulate the operon from outside the cell, represented by a rectangle around $M$ and $L$.}
    \label{lacop:wd}
\end{figure}


Boolean functions can be rewritten as polynomial functions over $\mathbb Z_2$ using the following translations: the Boolean expression $x \lor y$ can be represented as the polynomial $x+y+xy$, $x\land y$ as $xy$, and $\lnot x$  as $x+1$. Applying these rules to the above functions yields the finite dynamical system $f:\mathbb Z_2^4\rightarrow \mathbb Z_2^4$ where $f=(f_{x_1},f_{x_2},f_{x_3},f_{x_4})$ and each $f_{x_i}$ is a polynomial in the variables $x_1:=M$, $x_2:=L$, $x_3:=L_e$ and $x_4:=G_e$. 

\begin{align}
\label{lac-model}
\nonumber
f_{x_1}&= x_2x_3x_4 + x_2x_3 + x_2x_4 + x_3x_4 + x_2 + x_3 \\ \nonumber
f_{x_2}&=x_1x_3x_4 + x_1x_3 \\ 
f_{x_3}&=x_3\\ \nonumber
f_{x_4}&=x_4\\ \nonumber
\end{align}

\begin{figure}[h]
    \centering
    \includegraphics[width=11cm]{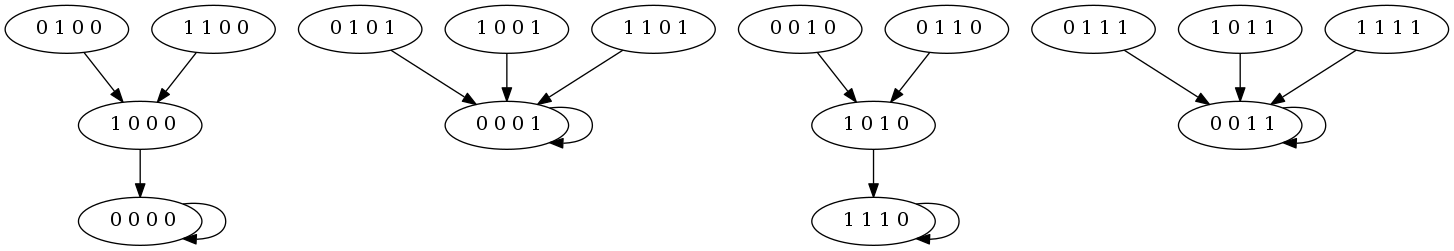}
    \caption{State space graph for the 4-dimensional finite dynamical system given by $f=(f_{x_1},f_{x_2},f_{x_3},f_{x_4})$. Each node is a state $(M, L, L_e,G_e)$ of the network and a directed edge from state $a$ to state $b$ indicates that $f(a)=b$.}
    \label{lacop:ss}
\end{figure}


Consider the first component of the state space of $f=(f_{x_1},f_{x_2},f_{x_3},f_{x_4})$ in Figure \ref{lacop:ss}: $$C_1=\{\{0, 0, 0, 0\}, \{0, 1, 0, 0\}, \{1, 0, 0, 0\}, \{1, 1, 0, 0\}\}.$$
Note that the data points in $C_1$ form a staircase.  By Corollary \ref{cor:shift}, the ideal $I(C_1)$ has a unique reduced Gr\"obner basis for any monomial order, namely
$$ G_1 = \{x_1^2+x_1, x_2^2+x_2, x_3, x_4 \}.$$
In particular the data set $C_1$ has the unique leading term ideal $L=\langle x_1^2,x_2^2,x_3,x_4\rangle$ and standard monomial basis $S=\{1,x_1,x_2,x_1x_2\}$ for any monomial order.

If we label the other components similarly,
$$C_2=\{\{0, 0, 0, 1\}, \{0, 1, 0, 1\}, \{1, 0, 0, 1\}, \{1, 1, 0, 1\}\},$$
$$C_3=\{\{0, 0, 1, 0\}, \{0, 1, 1, 0\}, \{1, 0, 1, 0\}, \{1, 1, 1, 0\}\},$$
$$C_4=\{\{0, 0, 1, 1\}, \{0, 1, 1, 1\}, \{1, 0, 1, 1\}, \{1, 1, 1, 1\}\},$$
we find that they are linear shifts of $C_1$, that is, $C_1\overset{\phi_{12}}{\sim} C_2$, $C_1\overset{\phi_{13}}{\sim} C_3$, and 
$C_1\overset{\phi_{14}}{\sim} C_4$,
where
\begin{align*}
\phi_{12}&= (x_1,x_2,x_3,x_4+1) \\
\phi_{13}&= (x_1,x_2,x_3+1,x_4)\\ 
\phi_{14}&= (x_1,x_2,x_3+1,x_4+1).
\end{align*}
According to Corollary \ref{cor:same}, the datasets $C_2$, $C_3$, and $C_4$ have the same leading term ideal and standard monomial basis as $C_1$.  
So each of $C_2$, $C_3$, and $C_4$ also has a unique reduced Gr\"obner basis. 
As such, each of the four data sets will identify a unique polynomial involving the monomials in $S=\{1,x_1,x_2,x_1x_2\}$. The advantage here is that each data set produces only one model; however since the polynomials in Equations (\ref{lac-model}) involve other monomials, we see that none of these data sets would correctly infer the original model.

\subsection{Using Linear Shifts to Compute Gr\"obner Bases}
In \citep{dimitrova2007grobner, dimitrova2014data}, the authors highlight the significance of being able to compute \textit{all} reduced Gr\"obner bases, as each one potentially corresponds to a distinct model.  Here we show how to use linear shifts to compute other Gr\"obner bases given one.  For example, given the Gr\"obner basis $G_1$ and linear shift functions $\phi_{12},\phi_{13},\phi_{14}$ from the previous section, we can directly apply the linear shift functions to produce the generators of the other reduced Gr\"obner bases explicitly, rather than computing them from the respective ideals.  

\begin{align*}
G_2=GB(I(C_2))&= \{x_1^2+x_1, x_2^2+x_2, x_3, x_4+1 \},\\
G_3=GB(I(C_3))&= \{x_1^2+x_1, x_2^2+x_2, x_3+1, x_4 \},\\ 
G_4=GB(I(C_4))&= \{x_1^2+x_1, x_2^2+x_2, x_3+1, x_4+1 \}.
\end{align*}

While algorithms (and their corresponding complexities) related to the above theoretical results are not in the scope of the presented work, we close with a note about its potential to greatly reduce the time to compute Gr\"obner fans of zero-dimensional ideals.  
The worst-case complexity of computing one Gr\"obner basis of a zero-dimensional ideal in a general setting is quadratic in the number of variables $n$ and cubic in the number of points $m$, that is $\mathcal O(nm^3+n^2m^2)$ \citep{abbott2000computing}, with various improvements in specialized settings.  
Computing a Gr\"obner fan for a zero-dimensional ideal from a given Gr\"obner basis is proved to be ``a polynomial-time algorithm in the size of the output'' \citep{fukuda}. 
In settings where data sets yield Gr\"obner fans with distinct cones, we can take advantage of  linear shifts: from one data set and its calculated fan (set of reduced Gr\"obner bases), use the linear shift functions to produce the reduced Gr\"obner bases for the ideals of the linearly shifted points.  We expect that finding linear shifts between data sets will have smaller complexity than computing multiple Gr\"obner fans; this analysis will be performed in future work.

\section{Discussion and Future Work}

We addressed the problem of characterizing data sets which correspond to ideals with a unique reduced Gr\"obner basis with respect to all monomial orders. Our results fill in important theoretical gaps in the use of polynomial dynamical systems as models in systems biology by providing a criterion for determining whether a set of data will give rise to a unique set of predictions (i.e., a unique model) without having to compute all reduced Gr\"obner bases associated with the data. Furthermore, this work decreases the computational cost of modeling using polynomial dynamical systems and has the potential to reduce the financial cost in collecting experimental data by identifying data sets that result in unambiguous hypotheses for future testing. 
This is especially important  in experimental settings since it is common for the number of observations ($m$) to be much smaller than the number of variables ($n$) in the system being studied. 

For example, the set $\{\{0, 0, 1, 0\}, \{0, 1, 0, 1\}, \{1, 0, 0, 1\}, \{1, 1, 0, 0\}, \{1, 1, 1, 1\}\}$ of 5 data points in $\mathbb Z_2$ has 13 distinct reduced Gr\"obner bases and 13 corresponding standard monomial bases.  If these are considered as inputs into the \textit{lac} operon as described in Section \ref{sec:lacop}, these standard monomial bases give rise to 4 distinct functions for $x_1$, 7 for $x_2$, 3 for $x_3$, and 5 for $x_4$, totaling 420 different finite dynamical systems, each with a different set of predictions that must be tested.
The number of  points (new experiments) that must be added to the existing data set to guarantee a unique model is an astonishing six: that is, the data set must be more than doubled! Instead, our work identifies candidate sets that are linear shifts of a staircase and have unique bases. (In fact, there are close to 600 sets with unique GBs out of nearly 4400 sets of size five; though there are many such sets, the chances of successfully picking such a set at random is about 14\%.)  

We note that partitioning data sets of same size into equivalence classes under the linear shift equivalence relation leaves many interesting question to explore. For example, knowing the number and size of equivalence classes as well as the distribution of the number of Gr\"obner bases across equivalence classes would provide valuable information for design of experiments and model inference.

Finally, while we were able to prove a sufficient condition for a set of points to have an ideal with a unique reduced Gr\"obner basis (Corollary \ref{cor:shift}), the condition is not necessary. A necessary and sufficient condition on the set of points which guarantees a unique reduced Gr\"obner basis for any monomial order would be of great value as it could lead to an algorithm for generating all such sets. These results are likely to also benefit other areas that use Gr\"obner bases such as neural ideals computation.

\section*{Acknowledgments}
This work was supported by the National Science Foundation under Awards DMS-1419038 and DMS-1419023.

\bibliographystyle{named}
\bibliography{thesis}

\begin{thebibliography}{}

\bibitem[\protect\citeauthoryear{Abbott \bgroup \em et al.\egroup
  }{2000}]{abbott2000computing}
J.~Abbott, A.~Bigatti, M.~Kreuzer, and L.~Robbiano.
\newblock Computing ideals of points.
\newblock {\em Journal of Symbolic Computation}, 30(4):341--356, 2000.

\bibitem[\protect\citeauthoryear{Adam \bgroup \em et al.\egroup }{2017}]{adam}
N.R. Adam, R.~Wieder, and D.~Ghosh.
\newblock Data science, learning, and applications to biomedical and health
  sciences.
\newblock {\em Ann N Y Acad Sci.}, 1387(1):5--11, 2017.

\bibitem[\protect\citeauthoryear{Babson \bgroup \em et al.\egroup
  }{2003}]{babson2003hilbert}
E.~Babson, S.~Onn, and R.~Thomas.
\newblock The {H}ilbert zonotope and a polynomial time algorithm for universal
  {G}r{\"o}bner bases.
\newblock {\em Advances in Applied Mathematics}, 30(3):529--544, 2003.

\bibitem[\protect\citeauthoryear{Cox \bgroup \em et al.\egroup }{1997}]{cox}
D.~Cox, J.~Little, and D.~O'Shea.
\newblock {\em Ideals, Varieties, and Algorithms}.
\newblock Springer Verlag, New York, 1997.

\bibitem[\protect\citeauthoryear{Curto \bgroup \em et al.\egroup
  }{2013}]{neuralring}
C.~Curto, V.~Itskov, A.~Veliz-Cuba, and N.~Youngs.
\newblock The neural ring: {A}n algebraic tool for analyzing the intrinsic
  structure of neural codes.
\newblock {\em Bulletin of Mathematical Biology}, 75:1571--1611, 2013.

\bibitem[\protect\citeauthoryear{Dalchau \bgroup \em et al.\egroup
  }{2018}]{dalchau}
N.~Dalchau, G.~Szép, R.~Hernansaiz-Ballesteros, C.P. Barnes, L.~Cardelli,
  A.~Phillips, and A.~Csikász-Nagy.
\newblock Computing with biological switches and clocks.
\newblock {\em Nat Comput.}, 17(4):761--779, 2018.

\bibitem[\protect\citeauthoryear{Dimitrova and
  Stigler}{2014}]{dimitrova2014data}
E.~Dimitrova and B.~Stigler.
\newblock Data identification for improving gene network inference using
  computational algebra.
\newblock {\em Bull. Math. Biol.}, 76(11):2923--2940, 2014.

\bibitem[\protect\citeauthoryear{Dimitrova \bgroup \em et al.\egroup
  }{2007}]{dimitrova2007grobner}
E.S. Dimitrova, A.S. Jarrah, R.~Laubenbacher, and B.~Stigler.
\newblock A {G}r{\"o}bner fan method for biochemical network modeling.
\newblock In {\em Proc. 2007 Internat. Symp. Symbolic Algebraic Computat.},
  pages 122--126. ACM, 2007.

\bibitem[\protect\citeauthoryear{Dimitrova \bgroup \em et al.\egroup
  }{2010}]{dimitrova2010}
E.S. Dimitrova, M.P. Licona, J.~McGee, and R.~Laubenbacher.
\newblock Discretization of time series data.
\newblock {\em J Comput Biol}, 17(6):853--868, 2010.

\bibitem[\protect\citeauthoryear{Dimitrova \bgroup \em et al.\egroup
  }{2019}]{robbiano-unique-gb}
E.S. Dimitrova, Q.~He, L.~Robbiano, and B.~Stigler.
\newblock Small {G}r\"{o}bner fans of ideals of points.
\newblock {\em Journal of Algebra and Its Applications}, 2019.

\bibitem[\protect\citeauthoryear{Dong}{2016}]{dong}
T.~Dong.
\newblock A two-dimensional improvement for {F}arr-{G}ao algorithm.
\newblock {\em J Syst Sci Complex}, 29:1382–--1399, 2016.

\bibitem[\protect\citeauthoryear{Farr and Gao}{2006}]{farr}
J.~Farr and S.~Gao.
\newblock Computing {G}r{\"o}bner bases for vanishing ideals of finite sets of
  points.
\newblock In M~Fossorier, H~Imai, S~Lin, and et~al., editors, {\em Applied
  Algebra, Algebraic Algorithms and Error-Correcting Codes}, page 118–127.
  Ann N Y Acad Sci., Springer, Berlin, 2006.

\bibitem[\protect\citeauthoryear{Fukuda \bgroup \em et al.\egroup
  }{2007}]{fukuda}
F.~Fukuda, A.~Jensen, and R.~Thomas.
\newblock Computing {G}r{\"o}bner fans.
\newblock {\em Mathematics of Computation}, 76(260):2189–--2212, 2007.

\bibitem[\protect\citeauthoryear{Garcia \bgroup \em et al.\egroup }{2018}]{gb}
R.~Garcia, L.D. Garc{\'i}a-Puente, R.~Kruse, J.~Liu, D.~Miyata, E.~Petersen,
  K.~Phillipson, and A.~Shiu.
\newblock Gr{\"o}bner bases of neural ideals.
\newblock {\em International Journal of Algebra and Computation}, 2018.

\bibitem[\protect\citeauthoryear{Goodwin}{1963}]{lac-goodwin}
B.~Goodwin.
\newblock {\em Temporal Organization in Cells}.
\newblock Academic Press, 1963.

\bibitem[\protect\citeauthoryear{He}{2016}]{he2016}
Q.~He.
\newblock {\em Algebraic Geometry Arising from Discrete Models of Gene
  Regulatory Networks}.
\newblock PhD thesis, Clemson University, 2016.

\bibitem[\protect\citeauthoryear{Jarrah \bgroup \em et al.\egroup
  }{2007}]{jarrah2007reverse}
A.S. Jarrah, R.~Laubenbacher, B.~Stigler, and M.~Stillman.
\newblock Reverse-engineering of polynomial dynamical systems.
\newblock {\em Adv. Appl. Math.}, 39(4):477--489, 2007.

\bibitem[\protect\citeauthoryear{Laubenbacher and
  Stigler}{2004}]{laubenbacher2004computational}
R.~Laubenbacher and B.~Stigler.
\newblock A computational algebra approach to the reverse engineering of gene
  regulatory networks.
\newblock {\em J. Theor. Biol.}, 229(4):523--537, 2004.

\bibitem[\protect\citeauthoryear{Lidl \bgroup \em et al.\egroup
  }{1997}]{lidl1997finite}
R.~Lidl, H.~Niederreiter, and G.-C. Rota.
\newblock Finite fields: Encyclopedia of mathematics and its applications.
\newblock {\em Comput. Math. Appl.}, 33(7):136--136, 1997.

\bibitem[\protect\citeauthoryear{Ozbudak \bgroup \em et al.\egroup
  }{2004}]{lac-ozbudak}
E.~Ozbudak, M.~Thattai, H.~Lim, B.~Shraiman, and A.~van Oudenaarden.
\newblock Multistability in the lactose utilization network of
  \textit{Escherichia coli}.
\newblock {\em Nature}, 427:737--740, 2004.

\bibitem[\protect\citeauthoryear{Santill\'{a}n}{2008}]{lac-santillian}
M.~Santill\'{a}n.
\newblock Bistable behavior in a model of the \textit{lac} operon in
  \textit{Escherichia coli} with variable growth rate.
\newblock {\em Biophysical Journal}, 94(6):2065--2081, 2008.

\bibitem[\protect\citeauthoryear{Schatz}{2015}]{schatz}
M.C. Schatz.
\newblock Biological data sciences in genome research.
\newblock {\em Genome Res.}, 25:1417--1422, 2015.

\bibitem[\protect\citeauthoryear{Stigler and Chamberlin}{2012}]{chamberlin}
B.~Stigler and H.M. Chamberlin.
\newblock A regulatory network modeled from wild-type gene expression data
  guides functional predictions in \textit{Caenorhabditis elegans} development.
\newblock {\em BMC Systems Biology}, 6(1), 2012.

\bibitem[\protect\citeauthoryear{Veliz-Cuba and Stigler}{2011}]{vcs}
A.~Veliz-Cuba and B.~Stigler.
\newblock Boolean models can explain bistability in the \textit{lac} operon.
\newblock {\em J Comput Biol.}, 18(6):783--794, 2011.

\bibitem[\protect\citeauthoryear{Wong \bgroup \em et al.\egroup
  }{1997}]{lac-wong}
P.~Wong, S.~Gladney, and J.~Keasling.
\newblock Mathematical model of the \textit{lac} operon: Inducer exclusion,
  catabolite repression, and diauxic growth on glucose and lactose.
\newblock {\em Biotechnology Progress}, 13(2):132--143, 1997.

\end{thebibliography}

\end{document}